\documentclass[12pt]{article}

\usepackage{geometry}
\usepackage{amsfonts, amsthm}
\usepackage{graphicx}
\usepackage{epstopdf}
\usepackage{mathtools}
\usepackage{pifont}
\usepackage{psfrag}

\allowdisplaybreaks

\usepackage{enumitem}
\usepackage{soul}
\setitemize{label=\small{\ding{117}}}
\setenumerate{label=(\alph*)}

\title{The  Radon Transform over Cones with Vertices\\on the Sphere and Orthogonal Axes}

\author{Daniela Schiefeneder and Markus Haltmeier}

\date{Department of Mathematics, University of
    Innsbruck\\Technikerstrasse 13, A-6020 Innsbruck, Austria\\
    { \tt \{Daniela.Schiefeneder,Markus.Haltmeier\}@uibk.ac.at}}

\DeclareMathOperator{\supp}{supp}
\DeclareMathOperator{\Kn}{\mathbf{K}}
\DeclareMathOperator{\In}{\mathbf{I}}
\newcommand{\fn}{\mathbf{f}}
\newcommand{\gn}{\mathbf{g}}
\newcommand{\edot}{\,\cdot\,}

\newcommand{\R}{\mathbb{R}}
\newcommand{\Z}{\mathbb{Z}}
\newcommand{\N}{\mathbb{N}}

\newcommand{\dom}[1]{\Delta(#1)}
\newcommand{\eps}{\epsilon}
\newcommand{\trans}{\mathsf{T}}

\newcommand{\K}{K}
\newcommand{\F}{F}
\newcommand{\f}{f}
\newcommand{\g}{g}

\newcommand{\mm}{m}

\newcommand*\dd{\mathop{}\!\mathrm{d}}
\newcommand\rmd{\mathrm{d}}
\newcommand{\ds}{\dd S}

\newcommand{\lk}{{\ell,k}}
\newcommand{\RC}{\mathcal{R}}
\newcommand{\bpm}{\begin{pmatrix}}
\newcommand{\epm}{\end{pmatrix}}

\newcommand{\Qm}{ Q}
\newcommand{\pp}{z}
\newcommand{\n}{\omega}
\newcommand{\al}{\alpha}
\newcommand{\ph}{\varphi}

\newcommand{\Cn}{C_\ell^{(n-2)/2}}

\newcommand{\sph}{\mathbb S}

\newtheorem{theorem}{Theorem}[section]
\newtheorem{lemma}[theorem]{Lemma}
\newtheorem{corollary}[theorem]{Corollary}

\newtheorem{alg}{Algorithm}

\theoremstyle{definition}
\newtheorem{definition}[theorem]{Definition}

\makeatletter
\newcommand*\bigcdot{\mathpalette\bigcdot@{.6}}
\newcommand*\bigcdot@[2]{\mathbin{\vcenter{\hbox{\scalebox{#2}{$\m@th#1\bullet$}}}}}
\makeatother

\newcommand\inner[2]{{#1}\bigcdot {#2}}
\newcommand\bfrac[2]{{#1/#2}}

\usepackage{todonotes}

\numberwithin{equation}{section}
\numberwithin{theorem}{section}

\newcommand{\bkl}[1]{\left(#1\right)}
\newcommand{\kl}[1]{(#1)}
\newcommand{\set}[1]{\{#1\}}
\newcommand{\abs}[1]{\lvert#1\rvert}
\newcommand{\norm}[1]{\lVert#1\rVert}

\usepackage{framed}
\begin{document}

\maketitle

\begin{abstract}
Recovering a  function  from its   integrals over  circular cones  recently gained significance  because of its relevance to novel medical imaging technologies such emission tomography using Compton cameras.
In this paper we investigate the case where the vertices of the cones of integration are restricted to a sphere in $n$-dimensional space and symmetry axes are orthogonal to the sphere.
We show invertibility of the considered transform and develop an inversion method based on series
 expansion and reduction to  a system of  one-dimensional integral equations of  generalized  Abel type.
 Because the arising kernels do not satisfy standard assumptions, we also develop a uniqueness
 result for generalized Abel equations where the kernel has zeros on the diagonal. Finally, we demonstrate
 how to numerically implement our inversion method and present numerical results.



\end{abstract}

\section{Introduction}
\label{sec:intro}

Many tomographic  imaging modalities are based on
the inversion of Radon transforms, which map a function onto its integrals over certain surfaces in $\R^n$ (see, for example, \cite{Kuc14,Nat01}).
The most basic example is the classical Radon transform
which maps the function onto its integrals over hyperplanes and which is the mathematical basis of X-ray CT.
Another well investigated  example is the  spherical Radon transform, which maps a function onto its integrals over hyper-spheres and which finds application in the recently developed photoacoustic tomography \cite{FinHalRak07,FinPatRak04,KucKun08}. In this article we consider   the conical Radon transform that maps a function to its integrals  over circular half cones.
The conical Radon transform  recently gained increased
interest, mainly due to its relevance for  SPECT using Compton cameras (see, for example, \cite{Allmaras13,BasZenGul98,CreBon94,Hal14a,JunMoo15,Moon16a,MorEtAl10,Par00,Smi05,TomHir02}).

\psfrag{B}{$\beta$}
\psfrag{w}{$\psi$}
\begin{figure}[tbh!]\centering
\includegraphics[width =\textwidth]{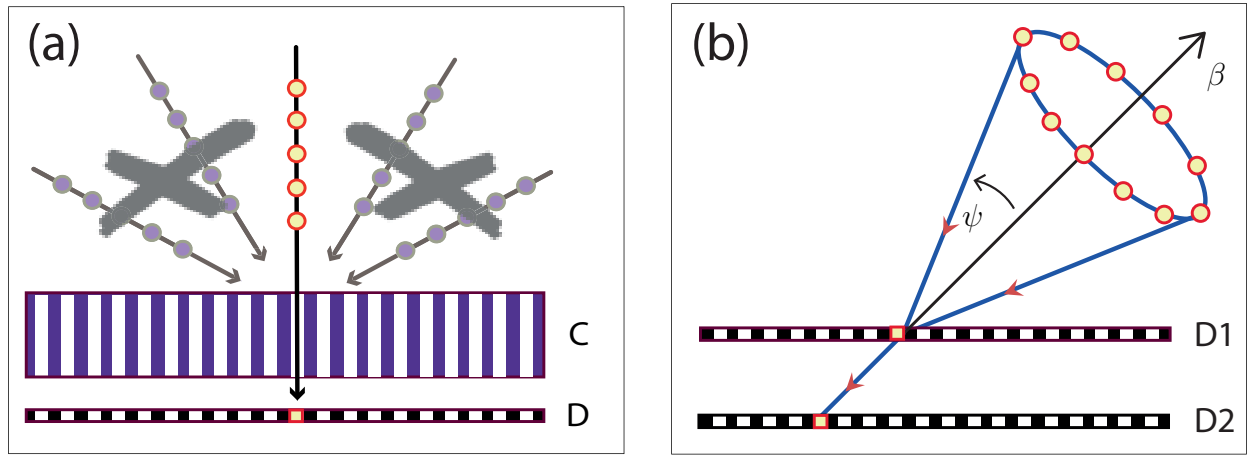}
\caption{(a) Standard gamma  cameras use collimators  which only observe  photons propagating orthogonal to the detector plane. The location of emitted photons  can be traced back to a straight line.  (b) A Compton camera consists of  two detector arrays  and any observed photon can be traced back to the surface of a cone.\label{fig:spect}}
\end{figure}

\subsection{SPECT using Compton cameras}

Single-photon emission computed tomography (SPECT) is a well established medical imaging technology for functional imaging.
In SPECT,  weakly radioactive tracers are given to patient and participate in the physiological processes. The  radioactive tracers can  be detected through the  emission of gamma ray photons which provide information   about the interior of patient. In order to obtain location information on the emitted photons,  the standard  approach in SPECT uses collimators which only record photons that enter the detector vertically.  As illustrated in Figure~\ref{fig:spect}(a),  such data provide values of line integrals of the tracer distribution.

A major drawback of using  collimators is that they remove most photons. Therefore the number  of recorded photons  is low  and the noise level high.
Typically, only one out of  $10 \, 000$ photons emitted from the patient is actually detected with this standard approach.
In order to increase the number of recorded photons, the concept of Compton cameras has been developed
 in~\cite{EveFleTidNig77,Sin83,TodNigEve74}.  As illustrated in Figure~\ref{fig:spect}(b), a Compton camera  consists of a scatter detector  array D1 and an absorption detector array D2.  A photon emitted in the direction of the camera undergoes
 Compton scattering in D1, and is absorbed in D2. The required distinction of individual  photons is obtained by coincidence detection.
Both detectors are position and energy sensitive,  and the measured energies can be used to determine the  scattering angle \cite{Sin83}. Using such information,
one concludes  that the detected photon must have been emitted on the surface of a circular cone, where the vertex is given by the
position  at  D1, the central axis  points from  the position on D2 to the position  on D1, and the opening angle equals the
 Compton scattering angle. Consequently, for  a distribution   of tracers, the Compton camera approximately provides
 integrals of  the marker distribution over conical surfaces.

\subsection{Inversion of the conical Radon transform}

As outlined above, SPECT with Compton cameras yields to the conical Radon transform
that  maps a function $f \colon \R^3 \to \R$ modeling the marker distribution to the surface  integrals $\int_{C(\pp, \beta, \psi)} \f  \ds$
over right circular half cones
\begin{equation*}
	C(\pp, \beta, \psi)
	=\{ \pp + r  \n \mid r \geq 0
	\text{ and }
	\n \in \sph^{2}
	\text{ with } \inner{\n}{\beta} =
	\cos\psi \} \,.
\end{equation*}
Here $\pp \in \R^3$ is the vertex of  the cone,  $\beta \in \sph^{2}$ the direction of the central axis, and $\psi \in (0, \pi/2)$
the  half opening angle. Variants of the conical Radon transform in $\R^2$  are known as  V-line or broken-ray transforms. These transforms appear in emission tomography with one-dimensional Compton cameras
\cite{BasZenGul97,JunMoo16}, or in the recently developed single scattering optical tomography~\cite{FloSchMar09}.
In this paper, we  consider the conical Radon transform in general dimension and further include a radial weight,
that can be adjusted to a particular application at hand. In previous work on Compton camera imaging~\cite{Smi05}, models with and without
radial weight have been proposed and used.

The conical Radon transform depends on six parameters $(\pp, \beta, \psi) \in \R^3 \times \sph^2 \times (0, \pi/2)$,
whereas the function $f$ only depends on three spatial coordinates.
Therefore  the problem of reconstructing $\f$ from its integrals over all circular cones  is highly overdetermined.
Several authors have studied the problem of inverting the function from integrals over particular subsets of all cones.
In SPECT with Compton cameras, the vertex is naturally fixed to the scattering surface D1. In the  case where D1 is a plane and the axis is fixed  to $\beta =(1,0,0)$, Fourier reconstruction formulas   have been derived in  \cite{CreBon94,NguTruGra05}. Formulas of the filtered backprojection type  have been derived in \cite{Hal14a,Moon16a}.    The case of variable axis and vertices  restricted to a surface   has been considered in  \cite{BasZenGul98,JunMoo15,Maxim2009,Par00,Smi05,Terzioglu15,TomHir02}.  See also
\cite{Allmaras13,AmbMoo13,FloMarSch11,GouAmb13} for related results on different conical transforms.

To the best of our knowledge, if the set of vertices is different from a plane and any vertex is associated with a single symmetry axis,
no results are known for reconstructing a function from its integrals over such cones.
In this paper we develop an inversion 
approach for the case when   $D1$ is a sphere and the symmetry axes of the cones are orthogonal to  the sphere.
We derive a reconstruction procedure based on spherical harmonics decomposition and show invertibility
of the considered transform. Spherical harmonics decompositions have been previously  used for studying other Radon transforms. See, for example,
\cite{Cor63,deans79,Lud66,Nat01} for the classical Radon transform, \cite{Quinto83} for a weighted
Radon transform over planes, \cite{Ambarsoumian10} for the circular Radon transform, or \cite{AmbMoo13,AmbRoy16} for a broken ray transform with vertices in a disc.
In these works,  the arising generalized Abel  equations satisfy all conditions needed in order to apply standard well-posedness results.
For the transform we study, one basic assumption of these results is violated, namely the kernels turn out to have zeros on the diagonal
(see Theorem \ref{thm:glk2}). Nevertheless, we are able to show solution uniqueness; see Theorem~\ref{thm:uni}.

\subsection{Outline}

The   paper is organized as follows.
In Section~\ref{sec:crt} we define the conical Radon transform
with vertices on the sphere  and orthogonal axis, and derive some elementary results for that transform.
Our main results are stated in Section~\ref{sec:main}.   By  using expansions in spherical
harmonics, we are able the reduce the conical Radon transform to a
set of explicitly given one-dimensional integral equations of the Abel type (see Theorem~\ref{thm:glk2}).
The invertibility of the one dimensional integral operators  will be given in  Theorem~\ref{thm:uni}. For that purpose, in Appendix~\ref{app:abel} we derive  uniqueness results for  first kind Volterra equations (Theorem~\ref{thm:volterra2}) and generalized  Abel equations with kernels having zeros on the diagonal (Theorem~\ref{thm:abel}). Theorem~\ref{thm:uni} in particular implies injectivity of the considered  transform and additionally yields an efficient  inversion method. In Section~\ref{sec:num}, we develop such a reconstruction procedure based on our theoretical findings and  present some numerical results. Finally, in  Section~\ref{sec:conclusion} we present a short summary  and discuss possible lines of future research.

\section{The  conical Radon transform} \label{sec:crt}

We start this section with defining the conical Radon transform that integrates  a function over cones with vertices on the unit sphere $\sph^{n-1} = \set{x \in \R^n \mid  \norm{x} =1}$ and central axis orthogonal to $\sph^{n-1}$.
For  $\pp \in \sph^{n-1}$  and $\psi\in (0, \pi/2)$, we denote by
\begin{equation*}
	C(\pp,\psi)
	=\{ \pp + r  \n \mid r \geq 0
	\text{ and }
	\n \in \sph^{n-1}
	\text{ with } -\inner{\n}{\pp} =
	\cos\psi \} \,,
\end{equation*}
the surface of a right circular half cone in $\R^n$ with
vertex $\pp  $, central axis $-\pp$ and half opening angle $\psi$.
We denote by $C^{\infty}_0(B_1(0))$ the set of all infinitely times differentiable functions $f \colon \R^n \to \R$ with  $\supp (f) \subseteq B_1(0) $, where   $B_1(0) \coloneqq \set{x \in \R^n \mid  \norm{x} <1}$ denotes  the unit ball in $\R^n$. Further, we denote by  $O(n) \subseteq \R^{n\times n}$  the set of all orthogonal  $n\times n $ matrices and set $e_1 \coloneqq (1, 0,  \ldots, 0)$.

\begin{definition}[The conical Radon transform $\RC_\mm\f$]
Let $\mm\in \Z$.
We define the conical Radon transform (with vertices on the sphere, orthogonal axis and weighting factor $\mm$)
of $f\in C^\infty_0(B_1(0))$ by
\begin{equation}\label{eq:crt}
\RC_\mm\f  \colon \sph^{n-1}  \times (0, \pi/2) \to \R
\colon
  (\pp, \psi) \mapsto   \int_{C(\pp,\psi)} f(x)\,\norm{x-\pp}^{\mm}\ds (x) \,.
\end{equation}
\end{definition}

The problem under study is recovering  the function $f$ from its conical Radon transform $\RC_\mm\f$.
We start by deriving explicit expressions for $\RC_\mm \f$.

\begin{lemma} Let\label{lem:R} $\mm \in \Z$ and $f\in C^\infty_0(B_1(0))$.
\begin{enumerate}
\item\label{lem:R-1} If $\Qm \in O(n)$  and $\pp \in \sph^{n-1}$, then  $(\RC_\mm\f)(\Qm\pp,\edot) = \RC_\mm(f \circ \Qm)(\pp,\edot)$.

\item \label{lem:R-2}
For every $(\pp, \psi) \in \sph^{n-1} \times (0, \pi/2)$, we have
\begin{align}
(\RC_\mm\f)(e_1, \psi )   \label{eq:Rf1}
      &= \int_0^2 r^\mm (r\sin (\psi))^{n-2}
  \\ & \nonumber \hspace{0.05\textwidth}
    \times \int_{\sph^{n-2}}  f  (1- r \cos(\psi), r \sin(\psi) \eta ) \ds(\eta)  \dd r \,,
\\ (\RC_\mm\f)(e_1,\psi) \label{eq:Rf2}
	&=
	\int_{0}^{\pi-\psi} \frac{(\sin(\psi))^{n-1}
    (\sin(\al))^{\mm+n-2}}{(\sin(\al+ \psi))^{\mm+n} }
  \\ & \nonumber \hspace{0.05\textwidth}
    \times \int_{\sph^{n-2}}
    \f \left(\frac{\sin(\psi)}{\sin(\al+\psi)} (\cos(\al), \sin(\al) \eta )\right)
     \ds (\eta) \dd \al \,.
\end{align}
\end{enumerate}
\end{lemma}

\begin{proof}
\ref{lem:R-1}  For every  $Q \in O(n)$ and every
$(\pp, \psi) \in \sph^{n-1} \times (0, \pi/2)$, we have
\begin{align*}
(\RC_\mm\f)( \Qm\pp,\psi)
	& = \int_{C(\Qm\pp,\psi)}~{f(x)\norm{x-\Qm\pp}^{\mm}\ds(x)}
\\      & = \int_{Q (C(\pp,\psi))}~{f(x)\norm{x-\Qm\pp}^{\mm}\ds(x)}
\\	& = \int_{C(\pp,\psi)}~{f(Q\,x)\norm{Q\,x-\Qm\pp}^{\mm}\ds(x)}
\\	& = \int_{C(\pp,\psi)}~{(f \circ \Qm)(x)\norm{x-\pp}^{\mm}\ds(x)}
\\	& = \RC_\mm (f \circ \Qm)(\pp,\psi).
\end{align*}

\ref{lem:R-2} Let $\Phi \colon  D  \to \R^{n-1}$ be any parametrization of $\sph^{n-2}$, where  $D\subseteq \R^{n-2}$ is an open subset of $\R^{n-2}$. Then
\[\Psi  \colon  D \times (0, \infty) \to \R^{n} \colon
(r, \beta) \mapsto (1-  r\cos (\psi), r \sin(\psi) \Phi(\beta)) \] is a parametrization of $C(e_1, \psi )$.  Elementary computation
shows that the Gramian determinant   of $\Psi$
is given by  $ \det ( \Psi'(r, \beta)^\trans  \Psi'(r, \beta)) = (r \sin \psi)^{2(n-2)} \det ( \Phi'(\beta)^\trans  \Phi'(\beta))$.
Consequently,
\begin{align*}
(\RC_\mm\f)(e_1, \psi )
	&= \int_{C(e_1,\psi)} \f(x)\,\norm{x-e_1}^{\mm}\ds (x)
\\      &= \int_0^\infty r^{\mm}  (r \sin (\psi))^{n-2}
                 \int_{D} \f  \left(1- r \cos(\psi), r \sin(\psi)
                 \Phi(\beta ) \right)\,\\
	&       \hspace{0.35\textwidth} \times \sqrt{\det ( \Phi'(\beta)^T  \Phi'(\beta))}\,  \dd \beta  \dd r
\\      &= \int_0^\infty r^{\mm}  (r \sin (\psi))^{n-2} \int_{\sph^{n-2}}
                 \f  \left(1- r \cos(\psi), r \sin(\psi) \eta \right)
                 \ds(\eta)  \dd r
 \,,
\end{align*}
which is \eqref{eq:Rf1}. Substituting $r  = \sin(\al) / \sin(\al+\psi)$, we have $\rmd r / \dd \al  = \sin (\psi) (\sin(\al+\psi))^{-2} $ and
$1- r \cos(\psi) = \cos (\al)  \sin (\psi) / \sin(\al+\psi) $; this yields \eqref{eq:Rf2}.
\end{proof}

Next we state the  continuity of $\f \mapsto \RC_\mm\f$ with respect to the  $L^p$-norms for $p \in \set{ 1,2}$. Similar results could of  course be obtained for any $p \in [1,\infty)$.

\begin{lemma}[Continuity\label{lem:Rcont} of $\RC_\mm$] Let $\mm \in \Z$, $f\in C^\infty_0(B_1(0))$ and $\eps \in (0,1)$.
\begin{enumerate}

\item\label{lem:Rc-1} If $2\mm+ n - 2  > 0$, then
$\norm{\RC_\mm\f}_{L^2} \leq \frac{\abs{\sph^{n-1}}\abs{\sph^{n-2}}~ 2^{2\mm+ n - 2}}{2\mm+ n - 2} \norm{\f}_{L^2}$.

\item\label{lem:Rc-2} If $\mm\geq 1$, then
$\norm{\RC_\mm\f}_{L^1} \leq 2^{\mm} \norm{\f}_{L^1}$.

\item\label{lem:Rc-3} If $\supp(f) \subseteq B_{1-\eps}(0)$, then $\norm{\RC_\mm\f}_{L^1} \leq C_{\eps, \mm}^{(1)} \norm{\f}_{L^1}$, $\norm{\RC_\mm\f}_{L^2} \leq C_{\eps, \mm}^{(2)} \norm{\f}_{L^2}$ for  constants $C_{\eps, \mm}^{(1)}$ and $C_{\eps, \mm}^{(2)}$ independent of  $\f$.
\end{enumerate}
\end{lemma}

\begin{proof}  \mbox{}
\ref{lem:Rc-1}
Let $\pp \in \sph^{n-2}$ and $\Qm \in O(n)$ satisfy  $\Qm e_1 = \pp$. By Lemma \ref{lem:R}, we have
 \begin{multline*}
	\norm{(\RC_\mm\f)(\pp, \edot )}^2_{L^2}
      = \int_{0}^{\pi/2}
          \abs{\RC_\mm(\f \circ \Qm )(e_1, \psi )}^2 \dd \psi
    =  \int_{0}^{\pi/2} (\sin (\psi))^{2(n-2)}
\\
	\times \biggl\lvert \int_0^2 r^{\mm+n-2}\int_{\sph^{n-2}}
	(\f \circ \Qm)  \left(1- r \cos(\psi), r \sin(\psi) \eta \right)
\ds(\eta)  \dd r \biggr\rvert^2 \dd \psi \,.
\end{multline*}
Using the  Cauchy-Schwarz inequality, we obtain
\begin{multline*}
\norm{(\RC_\mm\f)(\pp, \edot )}^2_{L^2} \leq
 \abs{\sph^{n-2}}
\biggl(\int_0^2  r^{2\mm+ n - 3} \dd r \biggr)
\biggl( \int_{0}^{\pi/2} (\sin (\psi))^{2(n-2)}
\\
\times  \int_0^2 \int_{\sph^{n-2}}
r^{n-1 } \abs{(\f \circ \Qm)  \left(1- r \cos(\psi), r \sin(\psi) \eta \right) }^2
\ds(\eta)  \dd r  \dd \psi \biggr)\,.
\end{multline*}
The first integral equals $\int_0^2  r^{2\mm+ n - 3} \dd r = 2^{2\mm+ n - 2}/(2\mm+ n - 2)$,
and the second can be bounded  by $ \norm{f}_{L^2}^2$.
Consequently, $\norm{(\RC_\mm\f)(\pp, \edot )}^2_{L^2}
\leq \abs{\sph^{n-2}} \, 2^{2\mm+ n - 2}/(2\mm+ n - 2) \norm{f}_{L^2}^2$.
 Integration over  $\pp \in \sph^{n-1}$ yields the claimed estimate.

\ref{lem:Rc-2}, \ref{lem:Rc-3}: Analogous to \ref{lem:Rc-1}.
\end{proof}

\section{Analytic inversion  of $\RC_\mm$}
\label{sec:main}

In this section, first we derive an  explicit decomposition of the conical Radon transform  in one-dimensional integral operators   (see Theorem~\ref{thm:glk2}). Second, we show the solution uniqueness of the corresponding generalized Abel equations (see Theorem~\ref{thm:uni}), which implies the invertibility of $\RC_\mm$.
For these results we will use the spherical harmonic
decompositions
\begin{align} \label{eq:fexp}
\f(r\theta)
	&=
	\sum_{\ell=0}^{\infty}
	\sum_{k=1}^{N(n,\ell)}
	\f_{\lk}( r)\,Y_{\lk} (\theta ) \,,
	\\ \label{eq:gexp}
    (\RC_\mm \f)(\pp,\psi)
	&=
	\sum_{\ell=0}^{\infty}
	\sum_{k=1}^{N(n,\ell)}
	(\RC_\mm \f)_{\lk} (\psi) Y_{\lk}(\pp) \,.
\end{align}
Here $Y_{\lk}$, for $\ell\in\N$ and $k \in \set{1, \dots,   N(n,\ell)} $,  denote spherical harmonics \cite{Mue66,seeley66}  of degree $\ell$ forming a  complete orthonormal system in $ \sph^{n-1}$.
The set of all $(\ell, k)$ with $\ell\in\N$ and $k \in \set{1, \dots,   N(n,\ell)}$ will be denoted by $I(n)$.

\subsection{Integral equations for $f_\lk$}

Let $C^\mu_\ell$ denote the Gegenbauer polynomials normalized in such  a way that $C^\mu_\ell(1)=1$. We derive three different relations between $f_{\lk}$ and $(\RC_\mm\f)_{\lk}$.
The first one is as follows.

\begin{lemma}\label{lem:glk}
Let   $\f \in C_0^\infty (B_1(0))$,
and let $\f_{\lk}$ and $(\RC_\mm\f)_{\lk}$ for  $ (\ell,k) \in I(n)$ be as in~\eqref{eq:fexp} and \eqref{eq:gexp}. Then
\begin{multline} \label{eq:glk}
	\forall \psi \in (0, \pi/2) \colon \quad
	(\RC_\mm\f)_{\lk}(\psi)
		=\abs{\sph^{n-2}}
	\int_0^{\pi-\psi}
	f_{\lk}\left(\frac{\sin(\psi)}{\sin(\alpha+\psi)}\right)
	\\
	\times \frac{(\sin(\psi))^{n-1}(\sin(\al))^{\mm+n-2}}{(\sin(\alpha+\psi))^{\mm+n}}
	\Cn(\cos(\al))\dd \alpha \,.
\end{multline}
\end{lemma}

\begin{proof}
Fix $ \pp\in \sph^{n-1} $ and let $ \Qm\in O(n) $ be any rotation with   $ \Qm e_1=\pp $.
Using the delta distribution  $ \delta $
and applying  the Funk-Hecke theorem, for any $\al\in (0, \pi)$  we have
\begin{multline}
\int_{\sph^{n-2}} Y_\lk( \Qm ( \cos(\al) , \sin(\al) \eta ) )\ds(\eta) \\
\begin{aligned}
	&=\int_{\sph^{n-1}}
	Y_\lk\left(\Qm \eta\right)\,\delta(\inner{e_1}{\eta}-\cos(\al))
	\,(1-(\inner{e_1}{\eta})^2)^{-(n-3)/2}\ds(\eta)
\\	&=
        \int_{\sph^{n-1}}
	Y_\lk (\eta)\,\delta(\inner{\pp}{\eta}-\cos(\al))
	\,(1-(\inner{\pp}{\eta})^2)^{-(n-3)/2}\ds(\eta)
\\	&=
        \abs{\sph^{n-2}}\,Y_\lk (\pp)
        \int_{-1}^1
	\,\delta(t-\cos(\al)) C_\ell	^{(n-2)/2}(t) \dd t	
\\	&=
        \abs{\sph^{n-2}}
        \, Y_\lk (\pp)
        \, C_\ell^{(n-2)/2}(\cos (\al)) \,.	
\end{aligned}
\end{multline}
Together with Lemma \ref{lem:R}, this yields
\begin{multline*}
\RC_\mm[x\mapsto  f_\lk (|x|)\,Y_{\lk}(x/\abs{x} )](\pp, \psi)
 =\abs{\sph^{n-2}}
\Bigg( \int_{0}^{\pi-\psi} f_{\lk}\left(\frac{\sin(\psi)}{\sin(\alpha+\psi)}\right)\\
\times
\frac{(\sin(\psi))^{n-1}
(\sin(\al))^{\mm+n-2}}{(\sin(\al+ \psi))^{\mm+n} }
 \, C_\ell^{(n-2)/2}(\cos (\al))  \dd \al\Bigg)
\,Y_\lk (\pp) \,.
\end{multline*}
The  linearity of $\RC_\mm$ gives~\eqref{eq:glk}.
\end{proof}

\begin{theorem}[Generalized Abel equation \label{thm:glk2} for $\f_\lk$]
Let   $\f \in C_0^\infty (B_1(0))$
and let $\f_{\lk}$ and $(\RC_\mm\f)_{\lk}$ be as \eqref{eq:fexp} and \eqref{eq:gexp} for  $ (\ell,k) \in I(n)$. Then, for  $\psi \in (0, \pi/2)$,
\begin{equation}\label{eq:glk2}
	(\RC_\mm\f)_{\lk}(\psi)
	=\abs{\sph^{n-2}} \sin(\psi)^{-\mm}
	\int_{\sin(\psi)}^{1} f_{\lk}(\rho)
	\frac{  \rho\, \K_\ell(\psi,\rho)}{\sqrt{\rho^2-(\sin(\psi))^2}}\dd \rho \,,
\end{equation}
with the kernel functions
\begin{multline}\label{eq:Kell}
\K_\ell(\psi,\rho) \coloneqq
\rho^{\mm+n-2} \sum_{\sigma = \pm 1}
\sigma^\ell    \sin\left(\arcsin\left(\sin(\psi)/\rho\right) -  \sigma \psi\right)^{\mm+n-2}
\\\times \Cn \left(\cos \left(\arcsin\left(\sin(\psi)/\rho\right) -  \sigma \psi\right) \right)    \,.
\end{multline}
\end{theorem}

\begin{proof}
We split the integral in Lemma \ref{lem:glk} in one  integral
over $ \alpha <  \pi/2 - \psi $ and  one over $\al \geq \pi/2 - \psi$.
For $\al < \pi/2 - \psi$ we substitute $\al = \arcsin(\sin (\psi)/\rho) - \psi$. We have $\rmd \al /\rmd \rho= - \sin(\psi) \, \rho^{-1} (\rho^2- \sin(\psi)^2)^{-1/2} $ and therefore
\begin{multline*}
\int_0^{\pi/2-\psi}
	f_{\lk}\left(\frac{\sin(\psi)}{\sin(\alpha+\psi)}\right)
	\,\frac{(\sin(\psi))^{n-1}(\sin(\al))^{\mm+n-2}}{(\sin(\alpha+\psi))^{\mm+n}}\,
	\Cn(\cos(\al))\dd \alpha
	\\
	\begin{aligned}
	&=
	 (\sin(\psi))^{n-1} \int_{\sin \psi}^1
	f_{\lk}(\rho)\,\Cn\left(\cos\left(\arcsin\left(\sin(\psi)/\rho\right) - \psi\right)\right)
	\\
	& \quad \times
	\left(\sin\left(\arcsin\left(\sin(\psi)/\rho\right) - \psi\right)\right)^{\mm+n-2}\frac{\rho^{\mm+n}}{\sin(\psi)^{\mm+n}}\,
	\frac{\sin (\psi) \dd \rho}{\rho \sqrt{\rho^2- (\sin(\psi))^2}}
\\    	&=
	(\sin(\psi))^{-\mm}
	\int_{\sin \psi}^1
	f_{\lk}(\rho)\Cn\left(\cos\left(\arcsin\left(\sin(\psi)/\rho\right) - \psi\right)\right)
	\\
	& \hspace{0.1\textwidth}\times
	  \left(\sin\left(\arcsin\left(\sin(\psi)/\rho\right) - \psi\right)\right)^{\mm+n-2}\frac{\rho^{\mm+n-1}  \dd \rho}{\sqrt{\rho^2- (\sin(\psi))^2}} \,.
\end{aligned}
\end{multline*}
In the case $ \alpha > \pi/2-\psi $,  we substitute
$\al =  \pi - \arcsin\left(\sin(\psi)/\rho\right)  - \psi$.
Repeating the above computations and using $C_\ell^\mu (-x) = (-1)^\ell  C_\ell^\mu (x)$ shows
\begin{multline*}
\int_{\pi/2-\psi}^{\pi-\psi}
	f_{\lk}\left(\frac{\sin(\psi)}{\sin(\alpha+\psi)}\right)
	\,\frac{(\sin(\psi))^{n-1}(\sin(\al))^{\mm+n-2}}{(\sin(\alpha+\psi))^{\mm+n}}
	\Cn(\cos(\al))\dd \alpha
	\\
	\begin{aligned}
    	&=
	(-1)^\ell
	(\sin(\psi))^{-\mm}
	\int_{\sin \psi}^1
	f_{\lk}(\rho)\Cn\left(\cos\left( \arcsin\left(\sin(\psi)/\rho\right)  + \psi\right)\right)
	\\
	& \hspace{0.1\textwidth}\times
	 \left(\sin\left( \arcsin\left(\sin(\psi)/\rho\right)  + \psi\right)\right)^{\mm+n-2}
	 \frac{\rho^{\mm+n-1}  \dd \rho}{\sqrt{\rho^2- (\sin(\psi))^2}} \,.
\end{aligned}
\end{multline*}
Together with \eqref{eq:glk}, this yields the claim.
\end{proof}

The relation between $f_\lk$ and $(\RC_\mm\f)_{\lk}$  given in Theorem~\ref{thm:glk2}
is   well suited for the numerical  implementation, see Section \ref{sec:num}.
For showing uniqueness of a solution, the following  equivalent form will be more appropriate.

\begin{lemma}
Let   $\f \in C_0^\infty (B_1(0))$\label{lem:glk3}
and let $\f_{\lk}$ and $(\RC_\mm\f)_{\lk}$ be as \eqref{eq:fexp} and \eqref{eq:gexp}.
Further, for every   $ (\ell,k) \in I(n)$  denote
\begin{enumerate}[label=(\alph*)]
\item\label{it:T1}
 $\hat{\g}_{\lk}(t)\coloneqq \abs{\sph^{n-2}}^{-1} (1-t)^{-(n-2)/2}
(\RC_\mm\f)_{\lk}(\arccos \sqrt{t})$;
\item \label{it:T2} $\hat{\f}_{\lk}(s)\coloneqq   f_{\lk}\left( \sqrt{1-s} \, \right) / 2$;
\item \label{it:T3} $\F_\ell(t,s) \coloneqq \sum_{\sigma = \pm 1} \sigma^\ell
\bkl{\sqrt{t}- \sigma \sqrt{t-s} \, }^{\mm+n-2}\Cn \left(\frac{\sqrt{t}\sqrt{t-s}+ \sigma (1-t)}{\sqrt{1-s}}\right) $.
\end{enumerate}
Then $\hat{\f}_{\lk}$ and $\hat{\g}_{\lk}$ are related via:
\begin{align}\label{eq:glk3}
	\forall t \in [0,1] \colon \quad
	\hat{\g}_{\lk}(t)=\int_{0}^{t} \hat{\f}_{\lk}\left(s\right)
	\frac{\F_\ell(t,s)}{\sqrt{t-s}}\dd s \,.
\end{align}
\end{lemma}

\begin{proof}
Substituting  $w\coloneqq\sin(\psi)$ in \eqref{eq:glk2} and using the trigonometric  sum  and difference identities shows
\begin{multline*}
\frac{1}{\abs{\sph^{n-2}}} (\RC_\mm\f)_{\lk} (\arcsin(w))
\\
\begin{aligned}
&= w^{-m} \int_{w}^{1} f_{\lk}(\rho)
\rho^{\mm+n-2} \sum_{\sigma = \pm 1}
\sigma^\ell    \sin\left(\arcsin\left(w/\rho\right) -  \sigma \arcsin(w)\right)^{\mm+n-2}
 \\  & \mbox{} \hspace{0.15\textwidth}
 	\times \Cn \left(\cos \left(\arcsin\left(w/\rho\right) -  \sigma \arcsin(w)\right) \right) \frac{\rho \dd \rho}{\sqrt{\rho^2- w^2}}
\\&=
w^{-m}\int_{w}^{1} f_{\lk}(\rho)
\rho^{\mm+n-2} \sum_{\sigma = \pm 1}
\sigma^\ell    \left(w/\rho \sqrt{1-w^2} -\sigma   w    \sqrt{1-w^2/\rho^2}\right)^{\mm+n-2}
 \\  & \mbox{} \hspace{0.15\textwidth}
 	\times \Cn \left(  \sqrt{1-w^2/\rho^2}\, \sqrt{1-w^2} + \sigma \, w^2/\rho \right) \frac{\rho \dd \rho}{\sqrt{\rho^2- w^2}}
\\&=
w^{n-2}  \int_{w}^{1} f_{\lk}(\rho)
 \sum_{\sigma = \pm 1}
\sigma^\ell    \left(\sqrt{1-w^2} -\sigma   \sqrt{\rho^2-w^2}\right)^{\mm+n-2}
 \\  & \mbox{} \hspace{0.15\textwidth}
 	\times \Cn \biggl(  \frac{\sqrt{\rho^2-w^2}\sqrt{1-w^2} + \sigma w^2}{\rho} \biggr)  \frac{\rho \dd \rho}{\sqrt{\rho^2- w^2}}
	 \end{aligned}
\end{multline*}
Next we set $w= \sqrt{1-t}$ and make the substitution  $\rho =\sqrt{1-s}$. Then we have
$1-w^2= t$, $\rho^2-w^2 = t-s$ and  $\arcsin (w) = \arccos (\sqrt{t})$, which shows
  \begin{multline*}
\frac{(1-t)^{-(n-2)/2}}{\abs{\sph^{n-2}}} (\RC_\mm\f)_{\lk} (\arccos( \sqrt{t})) =
  \frac{1}{2}\int_{0}^t f_{\lk}(\sqrt{1-s})
  \\
  \times  \sum_{\sigma = \pm 1}
\sigma^\ell    \left(\sqrt{t} -\sigma   \sqrt{t-s}\right)^{\mm+n-2}
\Cn \biggl(  \frac{\sqrt{t}\sqrt{t-s}  + \sigma (1-t)}{\sqrt{1-s}} \biggr)  \frac{\dd s}{\sqrt{t- s}} \,.
\end{multline*}
This together with  \ref{it:T1}-\ref{it:T3} yields \eqref{eq:glk3}.
\end{proof}

\subsection{Solution uniqueness}

Any of the integral equations \eqref{eq:glk3} is of  generalized Abel type.
Using the  symmetry  of the Gegenbauer polynomials,
we see that $ \F_\ell(s,s) = 2  \, s^{(\mm+n-2)/2}\, \Cn\left(\sqrt{1-s}\right)$.
Since the Gegenbauer polynomials have zeros  in  $[0,1]$, so has the function $s \mapsto \F_\ell(s,s)$.  Consequently, standard theorems on well-posedness  do not apply to \eqref{eq:glk3},
because such results require a non-vanishing diagonal.

To investigate unique solvability of~\eqref{eq:glk3} (and, as consequence,  of~\eqref{eq:glk2}),  we derive a  uniqueness result for generalized Abel   equations of the form
\begin{equation}\label{eq:Abel}
	\forall t \in [a,b] \colon \quad
	\int_a^t \frac{F(t,s)}{\sqrt{t-s}}\,\f(s)\dd s =g(t) \,,
\end{equation}
where $g \in C([a,b])$ corresponds to given data and $ F \in C(\dom{a,b})$,  with   $\dom{a,b} \coloneqq \set{(t,s) \in \R^2 \mid a \leq s \leq t\leq b }$, is a continuous kernel.

\begin{theorem}[Solution uniqueness\label{thm:abel} of Abel equations with kernel having  zeros on the diagonal]
Suppose that $ F\colon \dom{a,b}\to \R$, where $a < b$, satisfies the following:
\begin{enumerate}[label=(F\arabic*)]
	\item\label{thm:abel-1} $F\in C^3(\dom{a,b})$.
	\item\label{thm:abel-2} $N_F  \coloneqq  \set{ s \in [a,b) \mid F(s,s)=0}$
	is finite and consists of simple roots.
	\item\label{thm:abel-3} For every $s \in N_F$, the gradient $(\beta_1,\beta_2)
	\coloneqq \nabla F (s,s)$  satisfies
	\begin{equation}\label{eq:ineq:Abel}
		1+\frac{1}{2}\,\frac{\beta_1}{\beta_1+\beta_2}>0 \,.
	\end{equation}
\end{enumerate}
Then, for any $g \in C([a,b])$, equation~\eqref{eq:Abel} has at most one solution $f \in C([a,b])$.
\end{theorem}

\begin{proof}
See Appendix~\ref{app:abel}.
\end{proof}

To the best of our knowledge, Theorem~\ref{thm:abel} is new; we are not aware of similar results for generalized 
Abel equations with zeros in the diagonal of the kernel.
We derive this result by exploiting a well-posedness theorem due to Volterra and P\'{e}r\`{e}s for  first kind Volterra equations (see Lemma~\ref{lemma:volterra}) together with a standard procedure of reducing generalized Abel equations to Volterra integral equations of the first kind.
We now apply Theorem~\ref{thm:abel} to the integral equation~\eqref{eq:glk3}:

\begin{theorem}[Uniqueness \label{thm:uni}  of recovering $\f_\lk$] Suppose $\mm  > -(n+1)/2$. For any  $\f \in C_0^\infty (B_1(0))$ and any $(\ell,k) \in I(n)$, the spherical harmonic coefficient $\f_{\lk}$ of $\f$   can be recovered as the unique solution of
\begin{equation*}
	\forall \psi \in \bkl{0, \pi/2}\colon \quad
	(\RC_\mm\f)_{\lk}(\psi)
	=\abs{\sph^{n-2}} \sin(\psi)^{-\mm}
	\int_{\sin(\psi)}^{1} \f_\lk (\rho)
	\frac{ \rho \, \K_\ell(\psi,\rho) \dd \rho}{\sqrt{\rho^2-(\sin(\psi))^2}} \,,
\end{equation*}
with the kernel functions $\K_\ell$ defined by \eqref{eq:Kell}.
\end{theorem}

\begin{proof}
Let $f \in C_0^\infty (B_1(0))$   vanish outside a ball of Radius $1- a^2$.
According to Lemma \ref{lem:glk3}, it is sufficient to show that   \eqref{eq:glk3} has a unique
solution. To show that this is indeed the case, we apply Theorem~\ref{thm:abel} by verifying  that
$\F_\ell \colon  \dom{a, 1} \to \R$  satisfies conditions \ref{thm:abel-1}-\ref{thm:abel-3}.

Ad \ref{thm:abel-1}: Using the abbreviations $q \coloneqq \mm+n-2$ and  $C \coloneqq  \Cn$, the kernel
$\F_\ell$ can be written in the form
\begin{equation*}
\forall (t,s)  \in\dom{a, 1} \colon \;
\F_\ell(t,s)  =  \sum_{\sigma = \pm 1} \sigma^\ell
\bkl{\sqrt{t}- \sigma \sqrt{t-s} \, }^q  C \left(\frac{\sqrt{t}\sqrt{t-s}+ \sigma (1-t)}{\sqrt{1-s}}\right) \,.\end{equation*}
From this expression it is clear that  $\F_\ell$ is smooth on $\set{(t,s) \in \dom{a, 1} \mid t \neq s}$. Further, by using  $C(-x) = (-1)^\ell C(x)$ one sees that $\F_\ell$ is an even polynomial in $\sqrt{t-s}$. This shows that  $\F_\ell$ is also smooth on the diagonal $\set{(t,s) \in \dom{a, 1} \mid t = s}$.

Ad \ref{thm:abel-2}: Next, consider the restriction $ v(s)   \coloneqq \F_\ell(s,s)
=2  \, s^{q / 2} C (\sqrt{1-s})$ of the kernel to the diagonal.
As an orthogonal polynomial, $C$  has a finite number of isolated and simple roots. We conclude that the same holds true for  $v$.

Ad \ref{thm:abel-3}:
Let $s_0 \in [a, 1)$ be a zero of $v$ and set  $(\beta_1,\beta_2) \coloneqq \nabla \F_\ell (s_0,s_0)$.
Then \begin{equation}\label{eq:beta12}
	\beta_1+\beta_2
	=
	v'(s_0)=
	- \frac{s_0^\bfrac{q}{2}}{\sqrt{1-s_0}}
	C'\left(\sqrt{1-s_0}\right) \,.
\end{equation}
Next we compute $\beta_1 = (\beta_1+\beta_2)- \beta_2$. We have
\begin{align*}
\F_\ell(s_0,s_0-\eps)
&=  \sum_{\sigma = \pm 1} \sigma^\ell
\bkl{\sqrt{s_0}- \sigma \sqrt{\eps} \, }^q  C \left(\frac{\sqrt{s_0}\sqrt{\eps}+ \sigma (1-s_0)}{\sqrt{1-s_0+\eps}}\right)
\\&=
\sum_{\sigma = \pm 1} \sigma^\ell
\bkl{ s_0^{q/2}
- \sigma q s_0^{(q-1)/2}  \sqrt{\eps}
+ \frac{q(q-1)}{2}  s_0^{(q-2)/2}\eps }
\\& \quad \times
\Bigg( C'(\sigma \sqrt{1-s_0}) \frac{\sqrt{s_0}}{ \sqrt{1-s_0}}  \sqrt{\eps}
+ \Big(C''(\sigma \sqrt{1-s_0}) \frac{s_0}{2(1-s_0)}\\
&\mbox{} \hspace{0.25\textwidth}-\frac{\sigma}{\sqrt{1-s_0}} C'(\sigma \sqrt{1-s_0})\Big) \eps
\Bigg) + \mathcal O(\eps^2)
\\&=
\frac{s_0^{q/2}}{ \sqrt{1-s_0}} \bkl{- (2 q+1)  C'(\sqrt{1-s_0})  +
  \frac{s_0 C''(\sqrt{1-s_0})}{ \sqrt{1-s_0}} }
\eps + \mathcal O(\eps^2)  \,.
\end{align*}
Here for the last equality we used the symmetry properties $C'(-x)=(-1)^{\ell+1} C'(x)$ and $C''(-x)=(-1)^\ell C''(x)$ for the first and second  derivatives of the Gegenbauer polynomials.
Because $C$ is a solution of the differential equation
$$(1-x^2)\,C''(x) - (n-1)\, x\, C'(x)+\ell\,(\ell+n-2)\,C(x)=0$$
and $s_0$ is a zero of $t \mapsto C(\sqrt{1-t})$, we have the identity  $ s_0 C''(\sqrt{1-s_0})/\sqrt{1-s_0} = (n-1) C'(\sqrt{1-s_0})$. We conclude that  $-\beta_2 = \bkl{- 2 q   +n-2}  s_0^{q/2}
C'(\sqrt{1-s_0})/\sqrt{1-s_0}$. Together with \eqref{eq:beta12} we obtain
\begin{equation}\label{eq:beta1}
\beta_1 = \bkl{- 2q + n-3} \frac{s_0^{q/2}}{ \sqrt{1-s_0}}
C'(\sqrt{1-s_0}) \,.
\end{equation}
From \eqref{eq:beta12} and  \eqref{eq:beta1} it follows that
\begin{equation*}
1+ \frac{\beta_1}{2(\beta_1 + \beta_2)}
=
1+ \frac{2q - n + 3}{2}
= \mm  + \frac{n+1}{2} > 0 \,.
\end{equation*}
This shows \ref{thm:abel-3}. Consequently,  Theorem~\ref{thm:abel} implies that
$\hat \f_\lk$ is the unique  solution of  the integral equation \eqref{lem:glk3}.
\end{proof}

Theorem~\ref{thm:uni} immediately implies the following uniqueness result for the conical Radon transform  $\RC_\mm$.

\begin{corollary}[Invertibility\label{cor:uni} of  $\RC_\mm$] Suppose  $\mm  > -(n+1)/2$.
If $f_1, f_2   \in C_0^\infty (B_1(0))$ are such that   $\RC_\mm f_1 = \RC_\mm f_2$, then $f_1 =f_2$.
\end{corollary}

\begin{proof}
Let  $f  \in C_0^\infty (B_1(0))$ satisfy   $(\RC_\mm\f)_{\lk} = 0$ for all $(\ell,k) \in I(n)$.
According to Theorem~\ref{thm:uni}, the integral equation \eqref{eq:glk2} has the
unique solution   $f_\lk = 0$, which implies  $f = 0$. The linearity of $\RC_\mm$ gives
the claim.
\end{proof}

\section{Numerical implementation}
\label{sec:num}

Theorems~\ref{thm:glk2}  and~\ref{thm:uni}  are the basis of the following inversion method for the conical Radon transform $\RC_\mm$:
\begin{enumerate}
\item
Compute  the expansion coefficients $(\RC_\mm\f)_{\lk}$
 in \eqref{eq:gexp}.

\item
Recover $\f_{\lk}$ from $(\RC_\mm\f)_{\lk}$
by solving \eqref{eq:glk2}.

\item
Compute $ f(r\theta) = \sum_{(\ell,k)\in I(n)} \f_\lk(r) Y_\lk (\theta)$.
\end{enumerate}
In this section, we show how to implement this reconstruction procedure.  We restrict ourselves to  two spatial dimensions ($n=2$) and the case $\mm=0$; extensions to general cases  are straightforward.

\subsection{Basic procedure for numerically inverting the conical Radon transform}

In two  spatial dimensions, the conical Radon transform  with~$\mm=0$  can be written  in the form
\begin{equation}
(\RC \f) (\ph, \psi)
\coloneqq
 \sum_{\sigma = \pm 1}\int_0^\infty
\f \kl{(\cos(\ph),\sin(\ph))
- r(\cos(\ph-\sigma\psi),\sin(\ph-\sigma\psi))  }  \dd r \,.
\end{equation}
Because $\RC \f$ consists of integrals of $\f$ over V-shaped lines, the 2D version is  also known as the V-line  Radon transform.
In the 2D situation, the spherical harmonics expansion equals  the common Fourier series  expansion, and we obtain the following reconstruction  procedure:

\begin{framed}
\begin{alg}[Series expansion for inverting  the \label{alg:vrt} V-line transform]\mbox{}\\
\ul{Goal:} Recover $\f \colon \R^2 \to \R$ from the V-line transform  $\RC \f \colon [0, 2\pi] \times (0, \pi/2)\rightarrow\R$.
\begin{enumerate}[label=(S\arabic*)]
\item \label{alg:S1}
Compute
$g_\ell(s) \coloneqq \int_0^{2\pi} (\RC\f)(\al, \arcsin(s)) e^{-i \al \ell}\dd\al$.
\item  \label{alg:S2}
For all $\ell \in \Z$,  recover $\f_\ell$ by solving the Abel equation
\begin{equation}\label{eq:abel-2d}
\forall s\in [0,1]
\colon \quad g_\ell(s)=
	\int_{s}^{1} f_\ell(\rho)
	\frac{\rho \K_\ell(s,\rho)}{\sqrt{\rho^2 - s^2}}\dd \rho,
\end{equation}
with $\K_\ell(s,\rho) \coloneqq \sum_{\sigma = \pm 1}\sigma^\ell \cos\left(\ell\left(\arcsin (s/\rho) - \sigma \arcsin(s)\right)\right)$.

\item \label{alg:S3}
Evaluate $ f(r (\cos \al, \sin \al) ) = \frac{1}{2\pi}\sum_{\ell\in \Z}  \f_\ell(\rho) e^{i \ell \al}$.
\end{enumerate}
\end{alg}
\end{framed}

\medskip
In order to implement Algorithm~\ref{alg:vrt},
we suppose that we have given discrete data
\begin{equation*} 
     \gn[k, i] := \RC_\mm \f\kl{ \ph_k, \arcsin(s_i)}
      \quad \text{ for }  (k, i) \in
      \set{M/2, \dots,  M/2-1} \times \set{0, \dots N} \,.
 \end{equation*}
Here $\ph_k \coloneqq   2\pi(k-1)/M$ describe the discrete vertex positions and $s_i \coloneqq  i / N $ for $i \in \set{0, \dots, N}$    corresponds to the discretization of the half opening angles.
In  our   implementation, we discretize any step in Algorithm~\ref{alg:vrt}. For computing the  Fourier coefficients in Step
\ref{alg:S1} and for evaluating the Fourier series in Step~\ref{alg:S3}, we use the standard FFT algorithm. In Step~\ref{alg:S1}, the FFT algorithm outputs approximations to $g_\ell$ for $\ell \in \set{-M/2, -M/2+1,\dots, M/2-1}$, which are used as inputs for the second step.
The main issue  in the reconstruction procedure is implementing Step~\ref{alg:S2}, which consists in solving the integral equation \eqref{eq:abel-2d}. For that purpose we use product integration  method using the mid-point rule \cite{Linz,Plato12,Weiss71}, as outlined in the following subsection.

\subsection{The mid-point method for numerically solving~\eqref{eq:abel-2d}}

To apply the mid-point method to \eqref{eq:abel-2d} for any $\ell \in \Z$, one starts with the uniform discretization $s_i = i/N $ of the interval $[0,1]$. Evaluating \eqref{eq:abel-2d} at the discretization points yields
\begin{equation}\label{eq:mid1}
\forall i \in \set{0, \dots, N}\colon \quad
	g_\ell(s_i)
	=\sum_{j=i}^{n-1}
	\int_{s_j}^{s_{j+1}}
	f_\ell\left(\rho\right)\,\frac{\rho\, K_\ell (s_i,\rho)}{\sqrt{\rho^2-s_i^2}}\dd \rho \,.
\end{equation}
One  approximately evaluates   the  right hand side in  \eqref{eq:mid1} by replacing the  restriction of $ \rho \mapsto f_\ell(\rho) \, K_\ell(s_i,\rho)$ to $ [s_j, s_{j+1}]$ by the  function value at the mid-point  of the interval and computing the resulting integral exactly.  By setting $\rho_j \coloneqq (j+1/2)/N$, this yields
\begin{align}\label{eq:LGS:f}
\forall i \in \set{0, \dots, N}\colon \; g_\ell(s_i)
	&\simeq
	\sum_{j=i}^{N-1} w_{i,j}
	\K_\ell \kl{s_i,\rho_j}  f_\ell \kl{ \rho_j }\,,\\
	\nonumber
	w_{i,j} &\coloneqq  \int_{s_j}^{s{j+1}} \frac{\rho}{\sqrt{\rho^2 - s_i^2}}\dd \rho=
\frac{\sqrt{(j+1)^2  - i^2 }-\sqrt{j^2 - i^2}}{n} \,.
\end{align}
The mid-point rule defines  numerical approximations  $\fn_\ell[j] \simeq  \f_\ell\kl{\rho_j}$ by
requiring~\eqref{eq:LGS:f} to be exactly  satisfied with $\fn_\ell[j]$ instead of  $\f_\ell\kl{\rho_j}$.

Next we define
\begin{enumerate}
\item the discrete kernels $\Kn_\ell = (w_{i,j} \, K_\ell(s_i,\rho_j))_{i,j =0,\dots,N-1} \in\R^{N \times N}$;
\item the discrete data $\gn_\ell = (\g_\ell(s_0), \dots, \g_\ell(s_{N-1}) )^\trans \in \R^N$;
\item the discrete unknowns $\fn_\ell = (\fn_\ell[0], \dots, \fn_\ell[N-1] )^\trans \in \R^N$.
\end{enumerate}
The product integration method using the composite mid-point rule consists in the end in solving the following system of linear equations:
\begin{equation} \label{eq:Minv}
\text{Find} \quad \fn_\ell \in \R^N
\quad \text{such that} \quad
\gn_\ell  = \Kn_\ell  \fn_\ell   \,.
\end{equation}
The matrix $\Kn_\ell$ is triangular.
Therefore, in the case that $\Kn_\ell$ is non-singular and well conditioned, equation \eqref{eq:Minv}
can efficiently be solved by forward substitution.

\subsection{Regularization of the mid-point method}

Because the  kernel function $\K_\ell$ has zeros in the diagonal, the matrix $\Kn_\ell$ may have diagonal entries being exactly or at least  close to zero. As a consequence, solving the system  \eqref{eq:Minv} of linear equations is ill-conditioned.
In order to obtain a stable solution, regularization methods  have to be applied. We use the method of Tikhonov regularization for that purpose \cite{EngHanNeu96,Gro84,Han98,TikArs77}.  In this approach, regularized solutions are  defined  as solutions of the regularized normal equation
\begin{equation} \label{eq:tik}
 \left(\Kn_\ell^\trans  \Kn_\ell  + \lambda \In_N \right) \fn_\ell
 =\Kn_\ell^\trans \gn_\ell \,.
\end{equation}
Here $\In_N \in \R^{N \times N}$ is the identity matrix  and $\lambda>0$  is a regularization parameter.

The regularization parameter in \eqref{eq:tik} could be chosen in dependence on the index $\ell \in \set{-M/2, -M/2+1,  \dots, M/2-1}$, in combination with a data driven parameter selection rule. However,  the development of such strategies is outside the scope of this paper.  In our initial simulation presented below, we  take the regularization parameter  $\lambda$  simply as a user selected  constant. Nevertheless, we emphasize that  $\lambda$  has to be taken carefully as a trade of between stability of inverting $\Kn_\ell^\trans  \Kn_\ell  + \lambda \In_N$ and accuracy of approximating the pseudo-inverse of $\Kn_\ell$. Tikhonov regularization can be interpreted as  one  member of  filter based regularization methods based on singular value decomposition \cite{EngHanNeu96}. Instead of Tikhonov regularization, one could also use any other filter based regularization method for stabilizing the product integration method. For comparison purpose we  also implemented truncated singular value decomposition (SVD) for regularizing~\eqref{eq:Minv}.

\begin{figure}[htb!]
\includegraphics[width=0.48\textwidth]{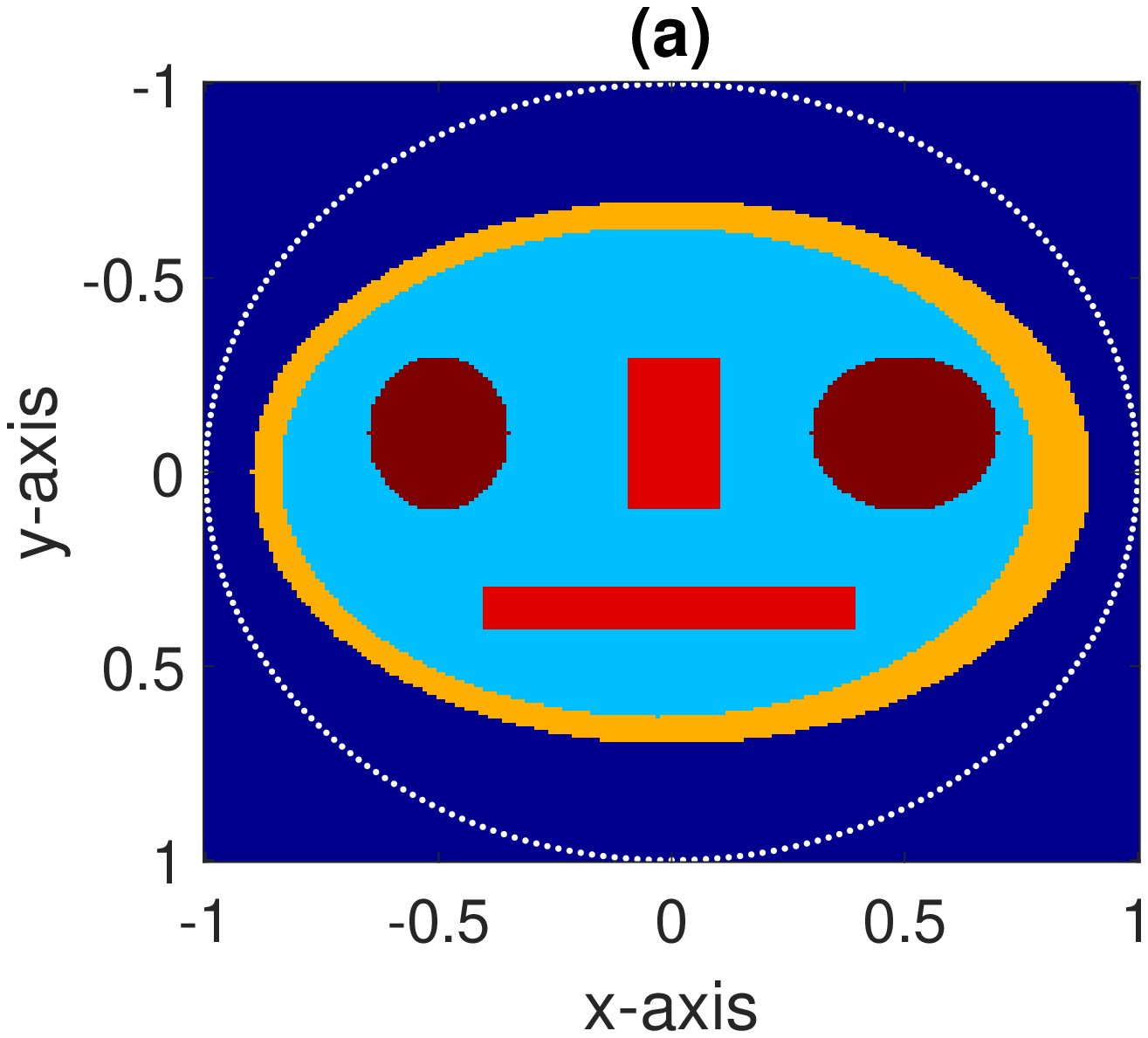}
\includegraphics[width=0.48\textwidth]{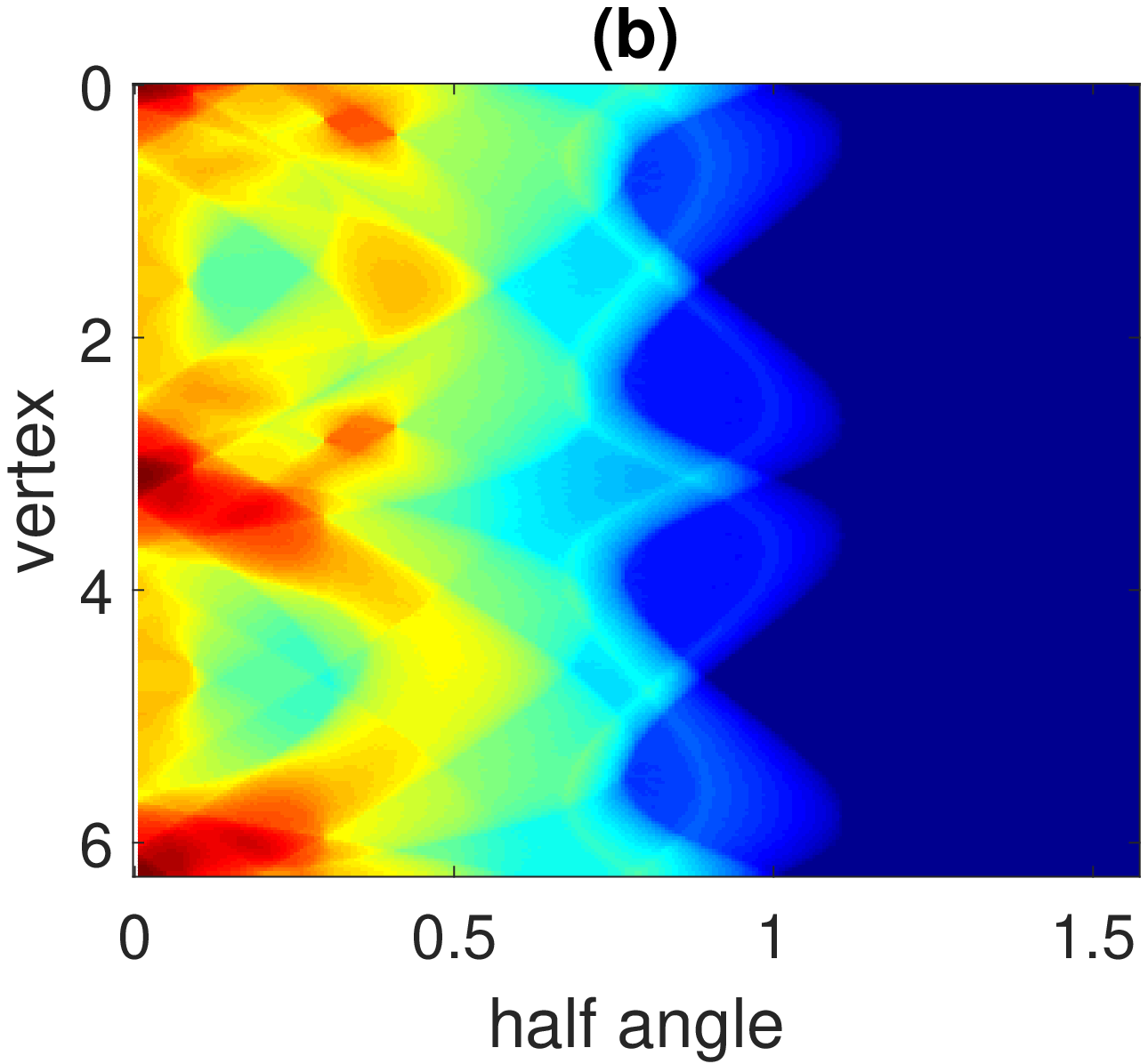}\\
\includegraphics[width=0.48\textwidth]{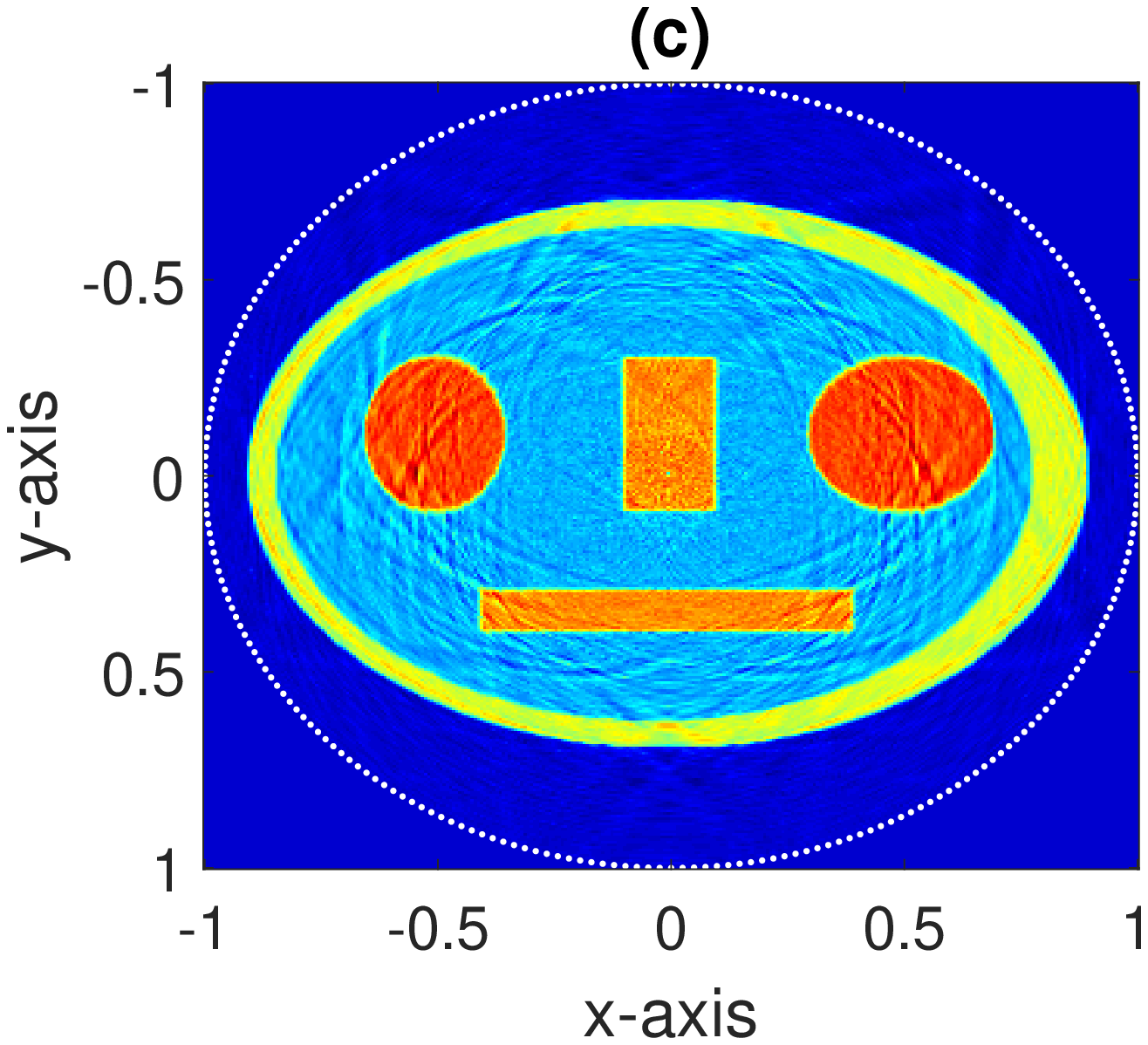}
\includegraphics[width=0.48\textwidth]{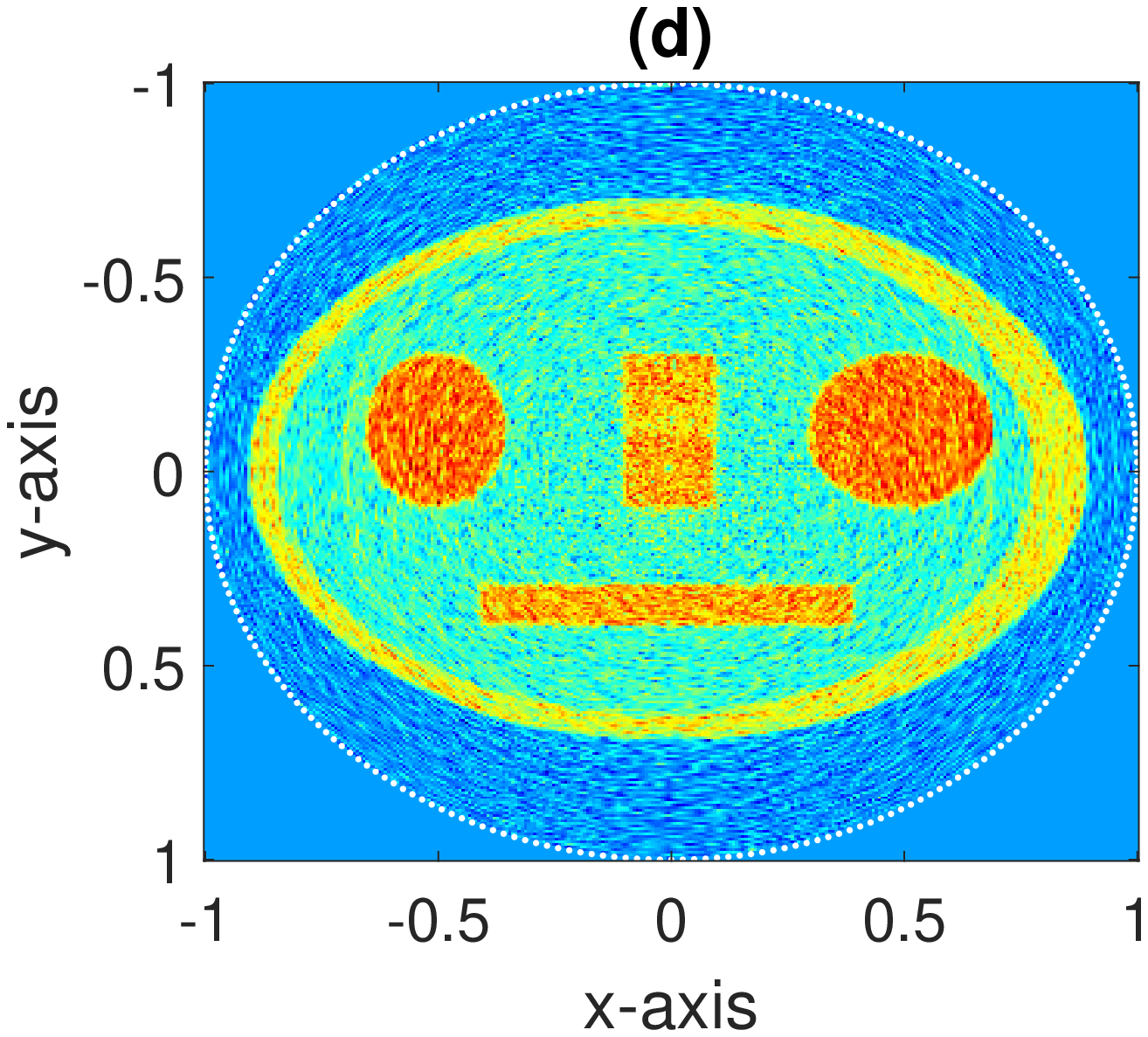}
\caption{\textsc{Reconstruction results:}\label{fig:num} (a) Smiley phantom  $\f$. (b) Simulated conical Radon transform  $\RC \f$. (c) Numerical reconstruction from simulated data using the derived algorithm. (d) Numerical reconstruction after adding Gaussian white noise with a relative $\ell^2$-error of $4 \%$.}
\end{figure}

For the case that the kernel is non-vanishing on the diagonal, the product integration  method \eqref{eq:Minv} using the mid-point rule is known to be convergent of order $3/2$; see \cite[Theorem 3.5]{Weiss71}.
Due to the zeros of the kernels, such results cannot be applied to the conical Radon transform. We are currently not aware  of any results for the (regularized) product integration method in that direction. Such investigations is an interesting  line of future research.

\subsection{Numerical example}

The reconstruction procedure outlined above has been implemented in \textsc{Matlab} and  tested on a  discretized version of a Smiley phantom shown in Figure~\ref{fig:num}(a)
  sampled on a Cartesian  $301 \times 301$ grid.
For  implementing the  conical Radon transform, we  numerically compute the integrals over V-lines using the composite trapezoidal rule. The numerically computed  V-line transform $\gn \in \R^{256 \times 301}$ using $M = 256$ vertex positions  and $301$ opening angles is  shown in Figure~\ref{fig:num}(b). The numerical reconstruction from such simulated data  using Algorithm~\ref{alg:vrt}   is shown in Figure~\ref{fig:num}(c).  The regularization parameter has been taken as $\lambda = 0.015$. We also tested our algorithm applied to noisy data $\gn+\mathbf{z}$, where $\mathbf{z} \in \R^{256 \times 301}$ is a realization of Gaussian white noise with $\norm{\mathbf{z}}_{\ell^2} /  \norm{\gn}_{\ell^2} \simeq 0.04$.
For noisy data, $\lambda = 0.05$ turned out to be a suitable regularization parameter.  In the resulting reconstruction, the structure of the phantom is
still clearly visible, although the noise has been amplified. Strategies for further improving the reconstruction quality will be investigated in future work.
Our numerical experiments  using truncated SVD led to  results very similar to Tikhonov regularization (not displayed). We remark that Tikhonov regularization is numerically more efficient because it only requires solving one linear equation for a symmetric positive definite matrix.

\section{Conclusion}\label{sec:conclusion}

In this  paper we studied the conical Radon transform $\RC_\mm$ that integrates a function in $\R^n$ over circular cones having  vertices on a sphere and axis orthogonal to the sphere including a radial weight $r^m$. By exploiting the spherical   symmetry of the problem, we have been able to decompose  $\RC_\mm$ in a product of explicitly computed one-dimensional integral equations of generalized Abel type.  By analyzing the zeros on the diagonal of the kernels  and exploiting a general uniqueness result developed in this paper, we have been able to show that any of these integral equations has a unique solution (provided $\mm >-(n+1)/2$). This  in particular implies the invertibility of $\RC_\mm$.

Based on our analytic  results, we developed a discrete reconstruction algorithm  where the main step is the numerical solution of the Abel type equations  involving the kernels $\K_\ell$. For that purpose, we applied the product integration method that yields to a linear matrix equation \eqref{eq:Minv}. Because  of the zeros of $s \mapsto \K_\ell(s,s)$,  equation~\eqref{eq:Minv} is ill-conditioned and has to be regularized, which has been done by  Tikhonov regularization. In future work we intend to investigate this issue by theoretically analyzing the  degree of ill-posedness of  the Abel integral equations  with kernels $\K_\ell$ and the stability of inverting $\RC_\mm$. We thereby also will consider convergence properties of the (regularized) product integration method. Further, it would be interesting to characterize the range of the involved Abel integral operators which finally might lead to a characterization of the range  of  $\RC_\mm$. Other interesting lines of research are  considering the conical Radon transform with non-orthogonal axis or deriving similar results for the case where the vertices are restricted to a cylindrical surface.

\appendix

\section{Uniqueness of Abel and first kind Volterra integral equations with kernels having zeros on the diagonal}
\label{app:abel}

In this appendix we prove Theorem~\ref{thm:abel}, a uniqueness result for generalized Abel equation.
For that purpose we first develop a uniqueness result for Volterra integral equations of the first kind (see Theorem~\ref{thm:volterra2}), 
that will subsequently be used to derive Theorem~\ref{thm:abel}.

\subsection{First kind Volterra integral equations}

For kernel $V \in C(\dom{a,b})$ and  data $g\colon [a,b]\to \R$, we consider the Volterra integral equation of first kind,
\begin{equation}\label{eq:volterra}
\forall u \in [a,b] \colon \quad
\int_a^u V(u,s) \, \f(s)\dd s =g(u)  \,.
\end{equation}
Standard results guaranteeing  existence and uniqueness  of a solution of   \eqref{eq:volterra} require $V(s,s)\neq 0$ for all $s\in[a,b]$.  Instead, we make use  of the following  non-standard result  that yields solution uniqueness in the case that the kernel has zeros on the diagonal.

\begin{lemma}[Theorem \label{lemma:volterra} of Volterra and P\'{e}r\`{e}s]
Equation \eqref{eq:volterra} has exactly one solution $\f \in C([a,b])$
if  $V$ and $g$ satisfy the  following:
\begin{enumerate}
\item\label{lem:v1} $s \mapsto V(s,s)$ has a simple root at $a$.
\item\label{lem:v2} $V(s,s)\neq 0$ for all $s\in(a,b]$.
\item\label{lem:v3} There  exist  $p_k \in C(\dom{a,b})$  with
$\partial_1 p_k \in C(\dom{a,b})$  for $k \in \set{0,1,2}$, and $\al_1,\al_2 \in \R$ with  $\alpha_1+\alpha_2\neq 0$ and  $1+ \alpha_1 / (\alpha_1+\alpha_2)>0$, such that
\begin{equation*}
	V(u,s)=
	\alpha_1 (u-a)+\alpha_2(s-a)+\sum_{k=0}^2 p_k(u,s) (s-a)^k (u-a)^{2-k}  \,.
	\end{equation*}
\item $g(s)=(s-a)^2 h(s)$ for some $h\in C^1 ([a,b])$.
\end{enumerate}
\end{lemma}

\begin{proof}
See~\cite{Feny3,volterra1936}.
\end{proof}

In the case that the kernel  $V$ has several zeros on the diagonal, we apply Lemma
 \ref{lemma:volterra}  to derive the  following Theorem~\ref{thm:volterra2}.
There  we only investigate uniqueness of solution, because in the exact data case the existence  of a solution is always guaranteed. Attempting to characterizing the range of the forward operators is an important aspect, that will be addressed in future work.

\begin{theorem}[Uniqueness result for first kind Volterra integral equations having several zeros in the diagonal]\label{thm:volterra2}
Suppose that $ V\colon \dom{a,b}\to \R$ satisfies the following:
\begin{enumerate}[label=(V\arabic*)]
	\item\label{it:volterra-1} $V\in C^3(\dom{a,b})$.
	\item\label{it:volterra-2} $N_V  \coloneqq  \set{ s \in [a,b) \mid V(s,s)=0}$
	is finite and consists of simple roots.
	\item\label{it:volterra-3} For every $s \in N_V$,
	$(\al_1, \al_2) \coloneqq \nabla V(s,s)$ satisfies
	$1+  \alpha_1/(\alpha_1+\alpha_2)>0$.
\end{enumerate}
Then, for every $g \in C([a,b])$, \eqref{eq:volterra} has at most one solution $\f \in C([a,b])$.
\end{theorem}

\begin{proof}
Write $N_V  = \set{s_0, s_1,  \dots , s_N}$ with $s_0 < s_1 < \cdots < s_N$ and assume that $s_0 =a$.
The case  $s_0>a$ can be  treated in a similar manner after showing solution uniqueness on $[a,s_0]$
using the  standard well-posedness result for non-vanishing diagonal. We will show recursively that $\f$ is  uniquely determined on $[a, s_{i+1}]$ by   \eqref{eq:volterra} for $i=0, \dots , N-1$.  For $i = 0$, consider the first kind Volterra equation
\begin{equation}  \label{eq:v-aux1}
\forall u \in [a, s_1] \colon
\quad \int_{a}^u  V(u,s) \f(s)\dd s
=
g_1(u)  \,,
\end{equation}
where $g_1  \coloneqq g|{[a, s_1]}$.
The assumptions made on $V$ imply that
$V|{\dom{a,b_1}}$ satisfies the conditions \ref{lem:v1}-\ref{lem:v3} in Lemma \ref{lemma:volterra}
for every $b_1< s_1$. Consequently, Lemma \ref{lemma:volterra} implies that   \eqref{eq:v-aux1} uniquely determines  $\f|{[a, b_1]}$. Taking the limit $b_1 \to s_1$ and using the  continuity of a possible solution shows that
$\f|{[a, s_1]}$ is uniquely defined.
Now suppose that $\f|{[a, s_i]}$ has already been shown to be  uniquely determined and consider the integral equation
$\int_{s_i}^{u} V(u,s) \f(s)\dd s
=  g_i(u)$ for $u \in [s_i, s_{i+1}]$, where
 $g_i(u) \coloneqq g(u) - \int_{a}^{s_i} V(u,s)\f(s)\dd s$.
Lemma \ref{lemma:volterra} applied to  the kernel $V|{\dom{s_i,b_{i+1}}}$ for $b_i \in (s_i, s_{i+1})$
and taking the limit $b_{i+1} \to s_{i+1}$ afterwards shows that $\f|{[a, s_{i+1}]}$ is uniquely determined.
\end{proof}

\subsection{Proof of  Theorem~\ref{thm:abel}}

We now derive Theorem \ref{thm:abel} as a consequence of Theorem~\ref{thm:volterra2}.   For that purpose, suppose that  $f \in C([a,b])$ is a solution of   \eqref{eq:Abel} with right hand side $g \in C([a,b])$ and kernel  $F \colon \dom{a,b} \to \R$ satisfying the assumptions \ref{thm:abel-1}-\ref{thm:abel-3} in Theorem \ref{thm:abel}.     By multiplying \eqref{eq:Abel} with $1/\sqrt{u-t}$, integrating  over $t$ and changing the order of integration, we obtain
\begin{equation} \label{eq:a2v}
\forall u \in [a,b] \colon \quad
\int_a^u \left( \int_s^u \frac{F(t,s)}{\sqrt{t-s}\sqrt{u-t}}\dd t \right)
\f(s)\dd s =\int_{a}^u \frac{g(t)}{\sqrt{u-t}}\dd t \,.
\end{equation}
The integral equation \eqref{eq:a2v} is a particular case of\eqref{eq:volterra} with continuous right hand side $u\mapsto \int_{a}^u g(t)/ \sqrt{u-t} \dd t$ and kernel $V$ defined by
\begin{equation*}
	V(u,s)\coloneqq
	\int_s^u \frac{F(t,s)}{\sqrt{t-s}\sqrt{u-t}}\dd t
	=\int_0^1 \frac{F(s+(u-s)r,s)}{\sqrt{r}\sqrt{1-r}}\dd r
	\,.
\end{equation*}
Consequently, \eqref{eq:Abel} has a unique solution if the kernel $V \colon \dom{a,b} \to \R\colon (u,s) \mapsto V(u,s)$ satisfies Items \ref{it:volterra-1}-\ref{it:volterra-3} in  Lemma \eqref{eq:volterra}.

\begin{itemize}
\item Ad \ref{it:volterra-1}: Because $F\in C^3(\dom{a,b})$, we have $V\in C^3(\dom{a,b})$.

\item Ad \ref{it:volterra-2}: For any $t \in [a,b]$ we have $V(s,s)=F(s,s)\int_0^1 1/\sqrt{r(1-r)}\dd r =\pi F(s,s)$. Consequently,  $N_V = N_F$ is finite and only consists of simple roots.

\item Ad \ref{it:volterra-3}:
Let $s_0$ be a root of $s \mapsto F(s,s)$ and let $(\beta_1,\beta_2) \coloneqq \nabla F (s_0,s_0)$ and    $(\al_1,\al_2) \coloneqq \nabla V (s_0,s_0)$. Then  $\alpha_1+\alpha_2 = \pi (\beta_1+\beta_2)$, and
\begin{multline*}
\alpha_1 =\partial_1 V(s_0,s_0)
=\int_0^1
	\frac{\left[\partial_1 F(s+(u-s)r,s)\right]_{u = s = s_0}}{\sqrt{r}\sqrt{1-r}}\dd r =
\\	=\beta_1 \int_0^1~{\frac{r}{\sqrt{r}\sqrt{1-r}}\dd r}=\frac{\pi}{2} \beta_1.
\end{multline*}
We conclude that $1+ \alpha_1 / (\alpha_1+\alpha_2)
=1+\beta_1/(2\beta_1+2\beta_2)$, which is positive according to the assumptions made on  the kernel $F$.
\end{itemize}
Consequently, Lemma~\ref{eq:volterra} implies that \eqref{eq:a2v} has a unique solution, which implies the  uniqueness of a solution of~\eqref{eq:Abel}.


\begin{thebibliography}{10}

\bibitem{Allmaras13}
M.~Allmaras, D.~Darrow, Y.~Hristova, G.~Kanschat, and P.~Kuchment.
\newblock Detecting small low emission radiating sources.
\newblock {\em Inverse Probl. Imaging}, 7(1):47--79, 2013.

\bibitem{Ambarsoumian10}
G.~Ambartsoumian, R.~Gouia-Zarrad, and M.~A. Lewis.
\newblock Inversion of the circular {R}adon transform on an annulus.
\newblock {\em Inverse Probl.}, 26(10):105015, 11, 2010.

\bibitem{AmbMoo13}
G.~Ambartsoumian and S.~Moon.
\newblock A series formula for inversion of the {V}-line {R}adon transform in a
  disc.
\newblock {\em Comput. Math. Appl.}, 66(9):1567--1572, 2013.

\bibitem{AmbRoy16}
G.~Ambartsoumian and S.~Roy.
\newblock Numerical inversion of a broken ray transform arising in single
  scattering optical tomography.
\newblock {\em IEEE Trans. Comput. Imaging}, 2(2):166--173, 2016.

\bibitem{BasZenGul97}
R.~Basko, G.~L. Zeng, and G.~T. Gullberg.
\newblock Analytical reconstruction formula for one-dimensional compton camera.
\newblock {\em {IEEE} Trans. Nucl. Sci.}, 44(3):1342--1346, 1997.

\bibitem{BasZenGul98}
R.~Basko, G.~L. Zeng, and G.~T. Gullberg.
\newblock Application of spherical harmonics to image reconstruction for the
  compton camera.
\newblock {\em Phys. Med. Biol.}, 43(4):887, 1998.

\bibitem{Cor63}
A.~M. Cormack.
\newblock Representation of a function by its line integrals, with some
  radiological applications.
\newblock {\em J. Appl. Phys.}, 34(9):2722--2727, 1963.

\bibitem{CreBon94}
M.~J. Cree and P.~J. Bones.
\newblock Towards direct reconstruction from a gamma camera based on compton
  scattering.
\newblock {\em IEEE Trans. Med. Imaging}, 13(2):398--407, 1994.

\bibitem{deans79}
S.~R.~R. Deans.
\newblock Gegenbauer transforms via the {R}adon transform.
\newblock {\em SIAM J. Math. Anal.}, 10(3):577--585, 1979.

\bibitem{EngHanNeu96}
H.~W. Engl, M.~Hanke, and A.~Neubauer.
\newblock {\em Regularization of inverse problems}, volume 375 of {\em
  Mathematics and its Applications}.
\newblock Kluwer Academic Publishers Group, Dordrecht, 1996.

\bibitem{EveFleTidNig77}
D.~B. Everett, J.~S. Fleming, R.~W. Todd, and J.~M. Nightingale.
\newblock Gamma-radiation imaging system based on the compton effect.
\newblock {\em Proc. {IEEE}}, 124(11):995--1000, 1977.

\bibitem{Feny3}
S.~Feny{\"o} and H.-W. Stolle.
\newblock {\em Theorie und {P}raxis der linearen {I}ntegralgleichungen 3},
  volume~76 of {\em Lehrb\"ucher und Monographien aus dem Gebiete der Exakten
  Wissenschaften (LMW). Mathematische Reihe}.
\newblock Birkh\"auser Verlag, Basel, 1984.

\bibitem{FinHalRak07}
D.~Finch, M.~Haltmeier, and Rakesh.
\newblock Inversion of spherical means and the wave equation in even
  dimensions.
\newblock {\em SIAM J. Appl. Math.}, 68(2):392--412, 2007.

\bibitem{FinPatRak04}
D.~Finch, S.~K. Patch, and Rakesh.
\newblock Determining a function from its mean values over a family of spheres.
\newblock {\em SIAM J. Math. Anal.}, 35(5):1213--1240, 2004.

\bibitem{FloMarSch11}
L.~Florescu, V.~Markel, and J.~Schotland.
\newblock Inversion formulas for the broken-ray {R}adon transform.
\newblock {\em Inverse Probl.}, 27(2):025002, 13, 2011.

\bibitem{FloSchMar09}
L.~Florescu, J.~C. Schotland, and V.~A. Markel.
\newblock Single-scattering optical tomography.
\newblock {\em Phys. Rev. E}, 79:036607, Mar 2009.

\bibitem{GouAmb13}
Rim Gouia-Zarrad and Gaik Ambartsoumian.
\newblock Exact inversion of the conical {R}adon transform with a fixed opening
  angle.
\newblock {\em Inverse Probl.}, 30(4):045007, 12, 2014.

\bibitem{Gro84}
C.~W. Groetsch.
\newblock {\em The Theory of Tikhonov Regularization for Fredholm Equations of
  the First Kind}.
\newblock Pitman, Boston, 1984.

\bibitem{Hal14a}
M.~Haltmeier.
\newblock Exact reconstruction formulas for a radon transform over cones.
\newblock {\em Inverse Probl.}, 30(3), 2014.

\bibitem{Han98}
P.~C. Hansen.
\newblock {\em Rank-Deficient and Discrete Ill-Posed Problems}.
\newblock SIAM Monographs on Mathematical Modeling and Computation. SIAM,
  Philadelphia, PA, 1998.

\bibitem{JunMoo15}
C.~Jung and S.~Moon.
\newblock Inversion formulas for cone transforms arising in application of
  {C}ompton cameras.
\newblock {\em Inverse Probl.}, 31(1):015006, 20, 2015.

\bibitem{JunMoo16}
C.~Jung and S.~Moon.
\newblock Exact inversion of the cone transform arising in an application of a
  compton camera consisting of line detectors.
\newblock {\em SIAM J. Imaging Sci.}, 9(2):520--536, 2016.

\bibitem{Kuc14}
P.~Kuchment.
\newblock {\em The {R}adon transform and medical imaging}.
\newblock SIAM, Philadelphia, 2014.

\bibitem{KucKun08}
P.~Kuchment and L.~A. Kunyansky.
\newblock Mathematics of thermoacoustic and photoacoustic tomography.
\newblock {\em Eur. J. Appl. Math.}, 19:191--224, 2008.

\bibitem{Linz}
P.~Linz.
\newblock {\em Analytical and numerical methods for {V}olterra equations},
  volume~7 of {\em SIAM Studies in Applied Mathematics}.
\newblock SIAM, Philadelphia, PA, 1985.

\bibitem{Lud66}
D.~Ludwig.
\newblock The {R}adon transform on euclidean space.
\newblock {\em Comm. Pure Appl. Math.}, 19:49--81, 1966.

\bibitem{Maxim2009}
V.~Maxim, M.~Frande{\c{s}}, and R.~Prost.
\newblock Analytical inversion of the {C}ompton transform using the full set of
  available projections.
\newblock {\em Inverse Probl.}, 25(9):095001, 21, 2009.

\bibitem{Moon16a}
S.~Moon.
\newblock On the determination of a function from its conical radon transform
  with a fixed central axis.
\newblock {\em SIAM J. Math. Anal.}, 48(3):1833--1847, 2016.

\bibitem{MorEtAl10}
M.~Morvidone, M.~K. Nguyen, T.~T. Truong, and H.~Zaidi.
\newblock On the {V}-line {R}adon transform and its imaging applications.
\newblock {\em Int. J. Biomed. Imaging}, 2010:208179, 6, 2010.

\bibitem{Mue66}
C.~M{\"u}ller.
\newblock {\em Spherical Harmonics}.
\newblock Lecture Notes in Mathematics. Springer Verlag, Berlin-New York, 1966.

\bibitem{Nat01}
F.~Natterer.
\newblock {\em The Mathematics of Computerized Tomography}, volume~32 of {\em
  Classics in Applied Mathematics}.
\newblock SIAM, Philadelphia, 2001.

\bibitem{NguTruGra05}
M.~K. Nguyen, T.~T. Truong, and P.~Grangeat.
\newblock Radon transforms on a class of cones with fixed axis direction.
\newblock {\em J. Phys. A}, 38(37):8003--8015, 2005.

\bibitem{Par00}
L.~C. Parra.
\newblock Reconstruction of cone-beam projections from compton scattered data.
\newblock {\em {IEEE} Trans. Nucl. Sci.}, 47(4):1543--1550, 2000.

\bibitem{Plato12}
R.~Plato.
\newblock The regularizing properties of the composite trapezoidal method for
  weakly singular {V}olterra integral equations of the first kind.
\newblock {\em Adv. Comput. Math.}, 36(2):331--351, 2012.

\bibitem{Quinto83}
E.~T. Quinto.
\newblock The invertibility of rotation invariant {R}adon transforms.
\newblock {\em J. Math. Anal. Appl.}, 91(2):510--522, 1983.

\bibitem{seeley66}
R.~T. Seeley.
\newblock Spherical harmonics.
\newblock {\em Amer. Math. Monthly}, 73(4):115--121, 1966.

\bibitem{Sin83}
M.~Singh.
\newblock An electronically collimated gamma camera for single photon emission
  computed tomography. part {I}: Theoretical considerations and design
  criteria.
\newblock {\em Med. Phys.}, 10(421):1983, 1983.

\bibitem{Smi05}
B.~Smith.
\newblock Reconstruction methods and completeness conditions for two compton
  data models.
\newblock {\em J. Opt. Soc. Am. A}, 22(3):445--459, 2005.

\bibitem{Terzioglu15}
F.~Terzioglu.
\newblock Some inversion formulas for the cone transform.
\newblock {\em Inverse Probl.}, 31(11):115010, 21, 2015.

\bibitem{TikArs77}
A.~N. Tikhonov and V.~Y. Arsenin.
\newblock {\em Solutions of Ill-Posed Problems}.
\newblock John Wiley \& Sons, Washington, D.C., 1977.

\bibitem{TodNigEve74}
R.~W. Todd, J.~M. Nightingale, and D.~B. Everett.
\newblock A proposed gamma camera.
\newblock {\em Nature}, 251:132--134, 1974.

\bibitem{TomHir02}
T.~Tomitani and M.~Hirasawa.
\newblock Image reconstruction from limited angle compton camera data.
\newblock {\em Phys. Med. Biol.}, 47(12):2129, 2002.

\bibitem{volterra1936}
V.~Volterra and J.~P\'{e}r\`{e}s.
\newblock {\em Th\'{e}orie g\'{e}n\'{e}rale des fonctionnelles}, volume~1.
\newblock Gauthier-Villars, 1936.

\bibitem{Weiss71}
R.~Weiss and R.~S. Anderssen.
\newblock A product integration method for a class of singular first kind
  {V}olterra equations.
\newblock {\em Numer. Math.}, 18:442--456, 1971.

\end{thebibliography}

\end{document}